\documentclass[reqno,12pt]{amsart}
\usepackage{amsmath,amssymb,amsthm,tikz,float,enumitem,mathtools}
\numberwithin{equation}{section}
\usepackage[margin=1in]{geometry}
\usepackage[bottom]{footmisc}

\usepackage[linkcolor=blue,colorlinks=true]{hyperref}
\usepackage{hyperref}
\usepackage[capitalise]{cleveref}

\usepackage[normalem]{ulem}

\allowdisplaybreaks

\def\ds{\displaystyle}
\def\p{\partial}

\subjclass[2010]{ 93B05, 93B07, 93B52, 93B60, 35L51}
\keywords{hyperbolic systems, control, observability, homogenization}
\makeatletter
\def\l@subsection{\@tocline{2}{0pt}{2pc}{5pc}{}}
\makeatother

\newcommand\numberthis{\addtocounter{equation}{1}\tag{\theequation}}

\newcommand{\mc}[1]{\mathcal{#1}}
\newcommand{\mb}[1]{\mathbb{#1}}

\newcommand{\rd}{{\rm d}}

\theoremstyle{definition}

	\newtheorem{theorem}{Theorem}
	
	\newtheorem{lemma}[theorem]{Lemma}
	\newtheorem{prop}[theorem]{Proposition}
	\newtheorem{defn}[theorem]{Definition}
	
	\newtheorem{rem}[theorem]{Remark}
	
	\numberwithin{theorem}{section}

\title[Control and homogenization of a hyperbolic system]{Control and homogenization of a coupled hyperbolic system with an oscillating coefficient}

\author{Vaibhav Kumar Jena}
\address{Vaibhav Kumar Jena
	\newline \indent
	{Department of Mathematics \newline \indent
	Indian Institute of Science \newline \indent
	Bengaluru, Karnataka, India.}}
\email{vaibhavkuma1@iisc.ac.in}

\author{Abu Sufian}
\address{Abu Sufian \newline \indent Department of Mathematics \newline \indent Indian Institute of Science Education and Research  \indent \newline \indent Tirupati, Andhra Pradesh, India. \indent }
\email{abusufian@iisertirupati.ac.in}

\begin{document}

\begin{abstract}
We study the control and homogenization properties of a coupled hyperbolic system characterised by a rapidly oscillating coefficient present in the principal term. Due to the high-frequency microstructural variations of the density parameter $\varepsilon$, classical uniform controllability results across all initial data fail under boundary controls alone. To overcome this limitation, we divide the spectrum into low and high-frequency components. Utilizing Ingham-Beurling-type inequalities combined with  spectral gap analysis, we demonstrate the uniform null controllability of the projected low-frequency and high-frequency states. Furthermore, we analyse the structural limits as $\varepsilon \to 0$ and prove that the associated sequence of boundary controls converge weakly to a null control for the underlying homogenized system. Finally, we show that by using an additional interior feedback control strategy, full space controllability can be achieved for all initial data in the energy space.
\end{abstract}

\maketitle

\section{Introduction}

The study of controllability properties for partial differential equations with rapidly oscillating coefficients has attracted considerable attention in recent years. Such problems arise naturally in the modelling of heterogeneous media, composite materials, and multi-scale physical systems, where the underlying parameters exhibit fine-scale spatial variations. A central question is to understand how these oscillations affect the ability to control the system and how the control properties behave in the homogenization limit. In this article, we are interested in studying the control problem for a coupled hyperbolic system with an oscillating density. The density depends on a parameter $\varepsilon$ and we study the behaviour of the system as $\varepsilon \rightarrow 0$.

This work generalises two previous results. First, the work in \cite{CC99} studies an analogous (boundary) control problem for a single one-dimensional wave equation. In particular, partial controllability results are obtained for the projection of the solutions over the subspaces generated by the eigenvectors with wavelengths shorter and longer than the parameter $\varepsilon$. Second, \cite{ped_per} considers a single wave equation and shows that to control all initial data in the space $L^2 (0,1) \times H^{-1} (0,1) $, one must put an additional (internal) feedback control. Our work is a generalisation of both \cite{CC99} and \cite{ped_per} to the case of coupled wave systems. We note that \cite{ped_per} gives explicit formulae for both states and controls and uses it in their analysis. However, for our case this is not possible because the coupling matrix is not known explicitly, which prevents us from obtaining such formulae, either for the state or the control. Nevertheless, we are able to show a positive controllability result.

\subsection{Problem}
Let $n \in \mathbb{N}$, and let $ A \in M_{n \times n} (\mathbb{R})$ be a matrix and $b \in \mb{R}^n$. 
We consider a function $a\in W^{2,\infty}(\mathbb{R})$, that is $1$-periodic, with the following uniform upper and lower bound
\begin{equation} \label{upper_lower_bound_1}
0<a_m\leqslant a(x)\leqslant a_M, \quad \forall x\in \mathbb{R}.
\end{equation}
For $\varepsilon \in (0,1)$, let us define the $\varepsilon$-periodic function $a^\varepsilon$ as $a^\varepsilon(x) := a\left(\frac{x}{\varepsilon}\right)$. For convenience, we will use the notation $\mc{L}^\varepsilon$ to denote the following operator
\[\mc{L}^\varepsilon := \p_x (a^\varepsilon(x)\p_x) I_{n\times n},\]
where $I_{n\times n}$ denotes the $n\times n$ identity matrix.
Furthermore, we write $u=(u_1,u_2,\cdots,u_n)^\text{tr} \in \mathbb{R}^n$, where tr denotes the transpose of the vector. Let $T>0$ and $f \in L^2(0,T)$. Consider the following evolutionary system
\begin{align}\label{general_B_M}
\begin{dcases}
\p_{tt} {u}^\varepsilon-\mc{L}^\varepsilon {u}^\varepsilon + A {u}^\varepsilon = 0, & (t,x)\in (0, T)\times (0,1), \\ 
{u}^\varepsilon (t,0)=0, \quad {u}^\varepsilon(t,1)= b f, & t \in (0,T),\\
{u}^\varepsilon (0,x)={u}^0(x), \quad \p_t {u}^\varepsilon (0,x) = {u}
^1(x), & x \in (0,1).
\end{dcases}
\end{align}
Here, $A$ is the coupling matrix and $b$ is the control vector.
We are interested in studying the controllability of the above system. We now provide the definition of controllability.
\begin{defn}
Given any $T>0$ and $({u}^0,{u}^1)\in (L^2(0,1))^n \times (H^{-1}(0,1))^n$, system \eqref{general_B_M} is said to be null controllable in time $T,$ if there exists a control function $f\in L^2(0,T)$, such that the solution $u^\varepsilon$ satisfies
\[u^\varepsilon(T,x)=0 \quad \text{and} \quad \p_t u^\varepsilon(T,x)=0, \quad \text{ for } x \in (0,1).\]
In this case, the function $f$ is said to be a null control for system \eqref{general_B_M}.
\end{defn}

Before presenting the main results, we mention the following hypotheses on the coupling matrix $A$ and the control vector $b$, which will be assumed throughout the article:
\begin{itemize}

\item[(H1)] The matrix pair $(A,b)$ satisfies the Kalman condition, i.e., the Kalman matrix $[A|b] = (b|Ab|A^2 b|\ldots|A^{n-1}b)$ has rank $n$. 

\item[(H2)] The matrix $A$ is symmetric and has $n$-distinct eigenvalues.

\item[(H3)] Let $\{ \lambda_k \}_{k=1}^\infty$ be the eigenvalues of the operator $-\mc{L}^\varepsilon$ on $(0,1)$ with homogeneous Dirichlet boundary conditions. Let $\theta_1, \ldots, \theta_n$ be the eigenvalues of the matrix $A$. Then the following property holds:
\[\max |\theta_i - \theta_j| < \min |\lambda_k - \lambda_l|, \quad \forall i\neq j, i,j = 1,2,\ldots,n \text{ and } k \neq l, k,l \in \mb{N}.\]


\end{itemize}

\begin{rem}
Assumption (H1) is a standard \emph{necessary} condition for controllability of finite-dimensional systems governed by the Kalman condition; see \cite{avd_ter}, for instance. In the context of distributed parameter systems, this condition ensures that the coupling matrix transmits the control action to all modes. Note that, putting $m$-controls, where $m < n$, is desirable, especially in applications. However, even in the absence of oscillation, there is no result that solves the control problem in the general case of $m$-controls; the relevant work \cite{avd_ter} considers $n$-controls.
\end{rem}

\begin{rem}
Assumption (H2) ensures that $A$ is self-adjoint and has a complete set of eigenvectors, which is essential for the spectral decomposition of the coupled system. The distinctness of eigenvalues allows for a clean separation of the spectral components and prevents degeneracies that would complicate the control construction.
\end{rem}

\begin{rem}
Assumption (H3) is technical but important. It ensures that the eigenvalues of the coupled operator are sufficiently separated so that we can put them in an increasing order and apply the Ingham-Beurling inequality; see \cite{beurling}. Without this assumption, the spectrum could intersect, making the spectral gap arguments invalid. An analogue of this assumption for the single wave equation can be found in \cite[Theorem 1.2]{avd_ter}, which allows them to use Ingham's inequality. Roughly speaking, the Ingham-Beurling inequality is a generalisation of Ingham's inequality.
\end{rem}

It is well known in the literature \cite{dol_rus} that the controllability of \eqref{general_B_M} is closely related to studying the corresponding adjoint system given by 
\begin{align} \label{eq_adjt}
\begin{dcases}
 \p_{tt} w^\varepsilon- \mc{L}^\varepsilon w^\varepsilon + A w^\varepsilon=0, & (t,x)\in (0, T)\times (0,1), \\ 
w^\varepsilon(t,0)=w^\varepsilon
(t,1)=0, & t \in (0,T),\\
w^\varepsilon(0,x) = w^0(x), \quad \p_t w^\varepsilon(0,x) = w^1(x), & x \in (0,1).
\end{dcases}
\end{align}
Related to the above system, we consider the following eigenvalue problem
\begin{equation} \label{eq_evp}
\begin{split}
-& \mc{L}^\varepsilon \Phi^\varepsilon + A \Phi^\varepsilon = \mu^\varepsilon \Phi^\varepsilon, \quad \text{in } (0,1),\\
& \Phi^\varepsilon(0) = \Phi^\varepsilon(1) = 0.
\end{split}
\end{equation}
It can be shown using similar ideas as in \cite[Section 3.1]{JS} that $\mu^\varepsilon := \mu^\varepsilon_{k,j} = \lambda_k^\varepsilon + \theta_j$, where $\lambda_k^\varepsilon$ is an eigenvalue for the problem
\begin{equation} \label{eq_evalue_L}
\begin{split}
- & \p_x(a^\varepsilon(x) \varphi^\varepsilon_k) = \lambda_k^\varepsilon \varphi_k^\varepsilon, \qquad \text{in } (0,1),\\
& \varphi(0) = \varphi(1) = 0, 
\end{split}
\end{equation}
and $\theta_j$ are the eigenvalues of $A$
\begin{equation} \label{eq_evalue_A}
A v_j = \theta_j v_j, \quad 0 \neq v_j \in \mb{R}^n.
\end{equation}
Notationally, we use $\Phi^\varepsilon_{k,j}$ to denote the eigenvector corresponding to $\mu^\varepsilon_{k,j}$, and furthermore, we have
\begin{equation} \label{eq_evec_1}
\Phi^\varepsilon_{k,j}=\varphi_k^\varepsilon\cdot v_j.
\end{equation} 
We refer the reader to \cite[Section 3.1]{JS} for the details. Furthermore, using \cite[Lemma 2.1]{avd_ter}, we can also choose the set of eigenvectors $\{v_j\}$ such that $\langle b,v_j \rangle = 1$, for all $j=1,2,\ldots,n$. This is a crucial property that is used in the spectral analysis; see the proof of Proposition \ref{lowfreobs}. Note that assumption (H3) is needed to ensure that the eigenspaces corresponding to $\mu_{k,i}$ and $\mu_{l,j}$ do not overlap for $(k,i)\neq (l,j)$ for all $k,l\in \mathbb{N},$ and $i,j \in \{1,2,\ldots,n\}$.

To present the main results, we describe the required frequency spaces. For $D>0$ and $\varepsilon > 0$, we define the low and high frequency spaces, respectively, as follows
\begin{equation*}
\begin{split}
L_{[D/\varepsilon]} := \text{span} \{ \Phi^\varepsilon_{k,l}: k \leq [D/\varepsilon], l = 1,2,\ldots,n \},\\
H^{[D/\varepsilon]} := \text{span} \{ \Phi^\varepsilon_{k,l}: k > [D/\varepsilon], l = 1,2,\ldots,n \}.
\end{split}
\end{equation*}
We denote by $\Pi_{L_{[D/\varepsilon]}}$ and $\Pi_{H^{[D/\varepsilon]}}$ the corresponding orthogonal projections from $(L^2(0,1))^n$ onto $L_{[D/\varepsilon]}$. Then, our main results addressing the controllability and homogenisation of system \eqref{general_B_M} are given below. For brevity, we present the results concerning the low frequency part here. The high frequency result is given in Theorem \ref{thm_highfreq}.

\begin{theorem}[Controllability] \label{thm_control_m}
Assume that (H1), (H2), and (H3) are satisfied. Also, let $T>2\sqrt{\int_0^1a(x)}$ and $[D/\varepsilon]\sim 1/\varepsilon$.
Let $({u}^0,{u}^1) \in (L^2(0,1))^n\times (H^{-1}(0,1)$$)^n$. Then for every $\varepsilon\in (0,1)$, there exists a control function $f^\varepsilon\in L^2(0,T)$  such that solution of system \eqref{general_B_M} satisfies
\begin{equation} \label{eq_proj_cont_low}
\left( \Pi_{L_{[D/\varepsilon]}} u^\varepsilon, \Pi_{L_{[D/\varepsilon]}} \p_t u^\varepsilon \right) = (0,0).
\end{equation}
Moreover, there exists a constant $C := C(T) >0$, independent of $\varepsilon$, such that
\begin{align*}
\| f^\varepsilon \|_{L^2(0,T)} \leqslant C(T) \|( \Pi_{L_{[D/\varepsilon]}}{u}^0,\Pi_{L_{[D/\varepsilon]}}{u}^1) \|_{(L^2(0,1))^n\times (H^{-1}(0,1))^n},\\
\left\|\left(\Pi_{L_{[D/\varepsilon]}} u^\varepsilon, \Pi_{L_{[D/\varepsilon]}} \p_t u^\varepsilon \right) \right\|_{L^\infty(0,T;(L^2(0,1))^n \times (H^{-1}(0,1))^n )}\leq C(T)\|f^\varepsilon \|_{L^2(0,T)}.
\end{align*}
\end{theorem}

\begin{theorem}[Homogenization] \label{thm_homogen_m}
Assume the hypothesis of Theorem \ref{thm_control_m}. 
Then, there exists a sequence of control functions $\{f^\varepsilon\}_{\varepsilon}\in L^2(0,1)$ for the system \eqref{general_B_M}, such that 
$$a^\varepsilon(1)f^\varepsilon\rightharpoonup a^0f_0 \text{ weakly } L^2(0,T), $$
where $f_0$ is a null control of the homogenized system of \eqref{general_B_M}, given by
\begin{align*} 
\begin{dcases}
\p_{tt} {u}-\mc{L}^0 {u} +A {u}= 0, & (t,x)\in (0, T)\times (0,1), \\ 
{u}(t,0)=0, \;{u}(t,1) = b f_0, & t \in (0,T),\\
{u} (0,x)={u}^0,\, {u}^1 (0,x)={u}^1 , & x \in (0,1),\\
{u}(T,x)=0, \quad {u}_t(T,x)=0, & x \in (0,1),
\end{dcases}
\end{align*}
and where the homogenized operator is $\mathcal{L}^0 := \left(a^0\partial_{xx} \right) I_{n\times n}= \left(\left(\ds \int_0^1\frac{1}{a(s)} \rd s\right)^{-1} \p_{xx}\right)I_{n\times n}$.
\end{theorem}

The main result concerning the controllability of all initial data in $L^2 (0,1) \times H^{-1} (0,1)$, using an additional feedback type control, is provided in Section \ref{sec_feedback}, as it is too technical to present in the introduction; see Theorem \ref{thm_full_space} and Theorem \ref{thm_full_space_hom}.

\subsection{Discussion}

We recall some background literature relevant to the problem studied in this work. Over the past several decades, the homogenization of control problems has attracted significant attention within the research community. One of the pioneering contributions in this area appears in \cite{zuazua94}, where the homogenization of an approximate controllability problem for a linear parabolic equation with rapidly oscillating diffusion coefficients was investigated. This line of research was later extended to settings involving perforated domains in \cite{don-nabil01}. Since then, numerous studies have addressed the homogenization of approximate controllability for evolutionary equations with oscillatory coefficients in more complex geometries, including domains with rough interfaces, multi-component structures, and perforations. For further developments and related results, we refer the reader to \cite{jose15, jose21, Fella19, PD99} and the references cited therein.

As far as controllability of waves without any oscillation is concerned, there is a huge amount of literature on it. For instance, see \cite{fur_iman, trig, xu, lucie, vkj1, vkj2, jsh}. However, the key difficulty in the current problem is that in the presence of an oscillation all the control results must be uniform with respect to the parameter $\varepsilon$, so that we can take the limit as $\varepsilon \rightarrow 0$. Hence, one must ensure that the control functions are estimated by the correct powers of $\varepsilon$. A similar problem in the parabolic setting was studied by the authors in \cite{JS}. There, the main challenge was likewise to determine the appropriate powers of $\varepsilon$ to ensure that the system remains well behaved as $\varepsilon \rightarrow 0$.

In the current article, we prove a variety of results related to the control and homogenisation of system \eqref{general_B_M}. First, we consider the control system and show that the low frequencies and high frequencies are controllable to the desired state, using only boundary control. This is done by proving a suitable estimate for the adjoint system \eqref{eq_adjt}, known as \emph{observability estimate} that looks like
\[ \int_0^{T} |\langle \p_x w^\varepsilon(t,1), b\rangle |^2 \rd t \geqslant C \| (w_0^\varepsilon,w_1^\varepsilon) \|^2_{(H^1(0,1))^n \times (L^2(0,1))^n},\]
for some $C>0$. We start from the observability inequality and apply the \emph{Hilbert Uniqueness Method} (HUM) developed in \cite{lionj:ctrlstab_hum, JL2}. Together, these yield the minimal \(L^2\)-norm control, which in turn provides a uniform bound on the control, with respect to $\varepsilon$.

More recently, uniform controllability has also been studied in higher dimensions; see \cite{FZ}. In that work, however, the system lacks zeroth-order coupling. The low-frequency uniform controllability is obtained for frequencies below \(\mc{O}( \varepsilon^{ -2/3 }) \). For related homogenization problems in controllability of evolutionary equations, we refer to \cite{teb12, zuazua94, Fella19, don-nabil01}. However, for the middle region we do not have controllability, along the similar lines as in \cite{CC99}.

Finally, we prove that to control all initial data in $(L^2(0,1))^n \times (H^{-1} (0,1))^n$, in addition to the boundary control, we also need to put an additional feedback control. This part is an analogue of \cite{ped_per}, where the authors discussed this result for a single equation.

\subsection{Outline}
In Section \ref{sec_spec}, we discuss some spectral properties for the operator $\p_x(a^\varepsilon \p_x)$ that are crucial to understand the properties of $\mc{L}^\varepsilon$. Here, we also prove the main boundary controllability results concerning the low and high frequencies. In Section \ref{sec_uniform}, we show that the obtained controls are uniform in $\varepsilon$. In Section \ref{sec_feedback}, we show that all initial data in $(L^2(0,1))^n \times (H^{-1}(0,1))^n$ can be controlled by putting an additional feedback type (interior) control.

\subsection{Acknowledgement}
VKJ was supported by the National Board for Higher Mathematics (NBHM), Department of Atomic Energy, Government of India, under the NBHM Postdoctoral Fellowship grant number 0204/16(7)/2024/R\&D-II/6758.

\section{Spectral estimates and controllability} \label{sec_spec}
In order to study the eigenvalue problem given in \eqref{eq_evp}, we recall a result from \cite{castro_zuazua,lopez_zuazua}, concerning the spectral gap for the following eigenvalue problem 
\begin{equation*}
\begin{split}
-\partial_x( a^\varepsilon(x) \partial_x \varphi^\varepsilon) = \lambda^\varepsilon \varphi^\varepsilon, & \ \text{ in } (0,1),\\
\varphi_\varepsilon(0)=\varphi_\varepsilon(1)=0.&
\end{split}
\end{equation*} 
For each $\varepsilon>0,$  there exists a sequence of spectral pairs $\{(\lambda_k^\varepsilon, \varphi_k^\varepsilon)\}_{k\in \mathbb{N}}$ such that 
\[ 0<\lambda^\varepsilon_1<\lambda_2^\varepsilon<\cdots<\lambda_k^\varepsilon<\cdots \to \infty,\]
and $\{\varphi_k^\varepsilon\}_{k\in \mathbb{N}}$ forms an orthonormal basis for $L^2(0,1)$. Also, by comparing $\lambda_k^\varepsilon$ with the spectrum of the operator $-\partial_{xx},$ it can be shown that (see \cite[Section 2.1]{lopez_zuazua}) $\lambda_k^\varepsilon\sim k^2$. Now, we recall two key results on the spectral pairs from \cite{CC99}.
\begin{prop} \label{eigenfunestilow}
Assume that $a\in L^\infty(\mathbb{R})$ is a periodic function. Given $\delta > 0$, there exists a constant $C(\delta) > 0$ such that 
$$\sqrt{\lambda_{n+1}^\varepsilon} - \sqrt{\lambda_n^\varepsilon} \geqslant \frac{\pi}{\sqrt{\overline{a}}} - \delta,$$
for all $n$ and $\varepsilon$ with $\varepsilon n \leqslant C(\delta)$ and where $\bar{a} = \int_0^1 a(s/\varepsilon) \rd s $. Furthermore, there exist $C_1, C_2 > 0$ such that the following estimates hold for the eigenfunctions $\varphi_n^\varepsilon$:
$$C_1 \int_0^1 |(\varphi_n^\varepsilon)'(x)|^2 \leqslant |(\varphi_n^\varepsilon)'(1)|^2 \leqslant C_2 \int_0^1 |(\varphi_n^\varepsilon)'(x)|^2,$$
where $'$ denotes derivative with respect to $x$ variable.
\end{prop}

\begin{prop}\label{eigenfunestihigh}
Let $a$ be a periodic function with $0 < a_m \leqslant a(x) \leqslant a_M < \infty$. Assume that $a \in W^{N+1, \infty}(\mathbb{R})$ for some $N \geqslant 1$. Given $\delta > 0$, there exists a constant $C > 0$ such that, if $n \geqslant C \varepsilon^{-1-1/N}$, we have
$$\sqrt{\lambda_{n+1}^\varepsilon} - \sqrt{\lambda_n^\varepsilon} \geqslant \frac{\pi}{\int_0^1 \sqrt{a(s/\varepsilon)} \, \mathrm{d}s} - \delta. $$
Furthermore, there exist $C_1, C_2>0$ such that the following estimates hold for the eigenfunctions $\varphi_n^\varepsilon$:
$$C_1 \int_0^1 |(\varphi_n^\varepsilon)'(x)|^2 \leqslant |(\varphi_n^\varepsilon)'(1)|^2 \leqslant C_2 \int_0^1 |(\varphi_n^\varepsilon)'(x)|^2. $$
\end{prop}

To obtain the homogenized null controllability of system \eqref{general_B_M}, a crucial step is to derive a sequence of null controls $ \{f^\varepsilon\}_{\varepsilon \in (0,1)} \in L^2(0,T) $ with uniform bounds \( \|f^\varepsilon\|_{L^2(0,T)} \leq C \), where $ C$ is a generic constant independent of $\varepsilon$. We refer to this property as \emph{uniform null controllability}. As discussed in the introduction, to prove controllability of \eqref{general_B_M}, it is enough to obtain an observability estimate of the form
\[ \|(w^\varepsilon_0,w_1^\varepsilon)\|_{(H^1(0,1))^n\times(L^2(0,1))^n}^2 \leqslant C(T) \int_0^{T} |\langle \p_x w^\varepsilon(t,1), b\rangle |^2 \rd t, \]
for some positive constant $C(T)$, where $w^\varepsilon$ is the solution of the adjoint system \eqref{eq_adjt}. Henceforth, we will work to show an estimate of the above type.

In Proposition \ref{eigenfunestilow}, choose $\delta = \frac{\pi}{2\sqrt{\bar{\bar{a}}}}$ and let $D = C(\frac{\pi}{2\sqrt{\bar{\bar{a}}}})$. Then we define the set of low frequencies as the set $\{\varphi_k^\varepsilon: k\leq[D /\varepsilon]\}$ and the set of high frequencies as the set $\{\varphi_k^\varepsilon :k > [D/\varepsilon]\}$.Using the low and high frequency spaces introduced above, we define the corresponding low and high frequency subspaces required for the present analysis as follows:
\begin{equation} \label{eq_low_freq_def}
L_{[D/\varepsilon]}=\text{span} \{\Phi_{k,l}^\varepsilon=\varphi_k^\varepsilon\cdot v_l: k \leq [D/\varepsilon]\, , l=1,2,\dots,n\}.
\end{equation} 
The set of low frequency spectrum is given by 
\[\{\mu_{k,l}^\varepsilon := \lambda_k^\varepsilon + \theta_l:  k \leq [D/\varepsilon],  l \in \{1,2,\ldots, n\}\},\]
where $\lambda_k^\varepsilon$ and $\theta_l$ are given by equations \eqref{eq_evalue_L} and \eqref{eq_evalue_A}, respectively.
Let $\theta_1 < \theta_2 < \ldots < \theta_n$ be the arrangement of the eigenvalues of $A$. Due to assumption (H3), we can write the spectrum in an increasing order as follows
\[\{ \mu_{1,1}^\varepsilon,\mu_{1,2}^\varepsilon ,\ldots,\mu_{1n}^\varepsilon ,\mu_{2,1}^\varepsilon, \ldots, \mu_{2n,\ldots}^\varepsilon\}.\]
By re-indexing, the above set can be written as $\{\mu_k:k \leq n[D/\varepsilon]\}$. 
Accordingly, the above and \eqref{eq_low_freq_def}, imply that
$$L_{[D/\varepsilon]}=\{\Phi_k^\varepsilon: k\leq n[D/\varepsilon]\}.$$First, we will consider the analysis for the low-frequency space.
Let $\Pi_{L_{[D/\varepsilon]}}$ denote the usual projection operator on $(H_0^1(0,1))^n$.

In the spectral estimates, we will need to apply the Ingham-Beurling inequality provided in \cite[Theorem 1.5]{beurling}. Hence, we must show that the our spectrum satisfies the hypothesis of Ingham-Beurling inequality. For this purpose, we have the following result.

\begin{lemma} \label{lemma_beurling}
There exists $M \in \mb{N}$ such that 
\[ \sqrt{\mu_{k+M}^\varepsilon} - \sqrt{\mu_k^\varepsilon} > C, \quad \forall k \in \mb{N},\]
for some $C>0$.
\end{lemma}

\begin{proof}
Using \eqref{eigenfunestilow}, we get that, if $N \in \mb{N}$, then
\begin{equation} \label{eq_beur1}
\sqrt{\lambda_{k+N}^\varepsilon} - \sqrt{\lambda_{k}^\varepsilon} = \sum_{i=1}^N \left( \sqrt{\lambda_{k+i}^\varepsilon} - \sqrt{\lambda_{k+i-1}^\varepsilon} \right) > N c_1,
\end{equation}
where $c_1 > 0$. Let us set $M=nN$, for some $N$ to be chosen later. Now, if $\mu_j^\varepsilon = \lambda_{k(j)}^\varepsilon + \theta_p$ for some $p \in \{1,2,\ldots,n\}$ , then 
\[ \mu_{j+M}^\varepsilon = \mu_{j+nN}^\varepsilon = \lambda_{{k(j)}+N}^\varepsilon + \theta_p,\]
where $k(j)$ depends on the index $j$ used in $\mu_j^\varepsilon$. We will write $k$ instead of $k(j)$ for convenience. Then, we have
\begin{align*}
\sqrt{\mu_{j+M}^\varepsilon} - \sqrt{\mu_{j}^\varepsilon} & = \sqrt{\lambda_{k+N}^\varepsilon + \theta_p} - \sqrt{\lambda_{k}^\varepsilon + \theta_p} \\
& = \frac{\left(\sqrt{\lambda_{k+N}^\varepsilon + \theta_p} - \sqrt{\lambda_{k}^\varepsilon + \theta_p}\right)}{\left(\sqrt{\lambda_{k+N}^\varepsilon} - \sqrt{\lambda_{k}^\varepsilon}\right)} \cdot \left( \sqrt{\lambda_{k+N}^\varepsilon} - \sqrt{\lambda_{k}^\varepsilon} \right) \\
& \simeq \mc{O}(1) \cdot \left( \sqrt{\lambda_{k+N}^\varepsilon} - \sqrt{\lambda_{k}^\varepsilon} \right)\\
& > \mc{O}(1) \cdot Nc_1.
\end{align*}
where we also used \eqref{eq_beur1} in the last step. Finally, choosing large enough $N$ so that the factor $Nc_1$ dominates the contribution coming from the $\mc{O}(1)$ term completes the proof of the lemma, with the choice $C = Nc_1$.
\end{proof}
Using the above gap condition, we will use the following version of the Ingham-Beurling inequality from \cite[Theorem 1.5]{beurling}.
\begin{lemma} \label{lemma_IB}
Let $\{\sigma_n\}_{n \in \mathbb{Z}}$ be a sequence of complex numbers. Suppose there exist $M \in \mathbb{N}$ and $\gamma > 0$ such that
\[ \sigma_{n+M} - \sigma_n \geq \gamma \quad \text{for all } n \in \mathbb{Z}. \]
Then, for any sequence $\{a_n\}_{n \in \mathbb{Z}} \in \ell^2(\mathbb{C})$ and for any $T > 0$, there holds
\[\int_0^T \left| \sum_{n \in \mathbb{Z}} a_n e^{i \sigma_n t} \right|^2 \rd t \asymp \sum_{n \in \mathbb{Z}} |a_n|^2, \]
where the constants implied by $\asymp$ depend only on $T$, $M$, and $\gamma$.
\end{lemma}

In the time interval $[0,T]$, we will prove the uniform null controllability for the low frequencies for system \eqref{general_B_M}. That is, we will show that, for each $\varepsilon>0$, there is a control  such that 
$$(\Pi _{L_{[D/\varepsilon]}} u^\varepsilon(T,\cdot),\Pi _{L_{[D/\varepsilon]}} \p_t u^\varepsilon(T,\cdot) )=0,$$
and the $L^2$-cost of the control is uniformly bounded with respect to $\varepsilon$.   
Now we will consider system \eqref{eq_adjt} with initial data $w^0, w^1 \in L_{[D/\varepsilon]}$.
By duality, it is enough to prove that the observability constant is independent of $\varepsilon$  for any $w^0,w^1 \in L_{[D/\varepsilon]}$. In particular we need to show the following result.
\begin{prop} \label{lowfreobs}
Let $w^\varepsilon =(w_1^\varepsilon,w_2^\varepsilon,\cdots, w_n^\varepsilon )^\text{tr}$ be the solution to the adjoint system \eqref{eq_adjt}  with given data $w^0,w^1  \in L_{[D/\varepsilon]}$. Then, given any $T>0,$ there exist positive constants $C_1(T), C_2(T)$ independent of $\varepsilon,$ such that 
\[C_1(T)\|(w^0,w^1)\|^2_{(H^1(0,1))^n\times(L^2(0,1))^n} \geqslant \int_0^{T} |\langle \p_x w^\varepsilon (t,1), b\rangle |^2 \rd t \geqslant C_2(T)\|(w^0,w^1)\|^2_{ (H^1(0,1))^n \times(L^2(0,1))^n}.\]
\end{prop}
\begin{proof}
While proving the above inequality, we fix $\varepsilon > 0$. Let $w^0, w^1 \in L_{[D/\varepsilon]}$. Then there exist sequences $\{a_k^\varepsilon\}_{|k|=1}^{n[D/\varepsilon]}$ and $\{b_k^\varepsilon\}_{|k|=1}^{n[D/\varepsilon]}$ in $\mathbb{R}$ such that
\[
w^0(x) = \sum_{|k|=1}^{n[D/\varepsilon]} a_k^\varepsilon \Phi^\varepsilon_k(x), 
\qquad
w^1(x) = \sum_{|k|=1}^{n[D/\varepsilon]} b_k^\varepsilon \Phi^\varepsilon_k(x).
\]
We express the solution $w^\varepsilon$ of the adjoint system \eqref{eq_adjt} as
\begin{equation}\label{eq_obs_hid_3}
w^\varepsilon(t,x) = \sum_{k=1}^{n[D/\varepsilon]} \sigma_k^\varepsilon(t)\, \Phi^\varepsilon_k(x),
\end{equation}
where the coefficients $\sigma_k^\varepsilon$ are to be determined.
Since $\Phi_k^\varepsilon$ is an eigenfunction of $-\mc{L}^\varepsilon + A$, substituting \eqref{eq_obs_hid_3} into the equation yields
\[
\ddot{\sigma}_k(t) + \mu_{k}^\varepsilon \sigma_k(t) = 0,
\]
with initial conditions
\[
\sigma_k(0) = a_k^\varepsilon, 
\qquad 
\sigma_k'(0) = b_k^\varepsilon,
\quad k = 1,2,\dots,n[D/\varepsilon].
\]
Solving this second-order ordinary differential equation, we obtain
\[
\sigma_k(t) = a_k^\varepsilon \cos(\sqrt{\mu_{k}^\varepsilon}\, t)
+ \frac{b_k^\varepsilon}{\sqrt{\mu_{k}^\varepsilon}} 
\sin(\sqrt{\mu_{k}^\varepsilon}\, t).
\]
Substituting the above into \eqref{eq_obs_hid_3}, we obtain
\begin{equation}\label{eq_obs_hid_4}
w^\varepsilon(t,x) = \sum_{k=1}^{n[D/\varepsilon]} \left( a_k^\varepsilon \cos(\sqrt{\mu_{k}^\varepsilon}\, t) + \frac{b_k^\varepsilon}{\sqrt{\mu_{k}^\varepsilon}}
\sin(\sqrt{\mu_{k}^\varepsilon}\, t) \right)\Phi^\varepsilon_k(x).
\end{equation}
Introducing the complex coefficients
\[ c_k^\varepsilon := \frac{1}{2} \left(a_k^\varepsilon - i \frac{b_k^\varepsilon}{\sqrt{\mu_{k}^\varepsilon}} \right) \quad c_{-k}^\varepsilon := \frac{1}{2} \left(a_k^\varepsilon + i \frac{b_k^\varepsilon}{\sqrt{\mu_{k}^\varepsilon}} \right), \]
and defining
\[ \rho_k := \sqrt{\mu_{k}^\varepsilon}, \qquad \rho_{-k} := -\sqrt{\mu_{k}^\varepsilon},\]
we can rewrite \eqref{eq_obs_hid_4} as
\[ w^\varepsilon(t,x) = \sum_{k=-n[D/\varepsilon]}^{n[D/\varepsilon]} c_k^\varepsilon e^{i \rho_k t} \Phi^\varepsilon_k(x).\]
This implies that
\begin{align*}
\int_0^T |\langle \p_x w^\varepsilon(t,1), b\rangle |^2 \rd t = \int_0^T \left| \sum_{k= -n[D/\varepsilon] }^{n[D/\varepsilon]} c_k^\varepsilon \langle (\Phi_k^\varepsilon )'(1), b\rangle e^{i \rho_k t} \right|^2 \rd t.
\end{align*}
Now, due to Lemma \ref{lemma_beurling}, we can apply Lemma \ref{lemma_IB}. Then, we get that there exist constants $C_1(T), C_2(T) > 0$ such that
\begin{align*}
C_1(T) \sum_{k=-n[D/\varepsilon] }^{n[D/\varepsilon]} 
|c_k^\varepsilon|^2 |\langle \Phi_k^\varepsilon)'(1),b\rangle|^2  & \le \int_0^T \left| \sum_{k=-n[D/\varepsilon]}^{n[D/\varepsilon]} c_k^\varepsilon |\langle (\Phi_k^\varepsilon)'(1),b\rangle |^2 e^{i \rho_k t} \right|^2 \rd t \numberthis \label{1628}\\
& \le C_2(T)\sum_{k=-n[D/\varepsilon]}^{n[D/\varepsilon]} 
|c_k^\varepsilon|^2 |\langle (\Phi_k^\varepsilon)'(1),b\rangle |^2.
\end{align*}
Now observe that
\[ \sum_{k=-n[D/\varepsilon]}^{n[D/\varepsilon]} |c_k^\varepsilon|^2 |\langle (\Phi_k^\varepsilon)'(1),b\rangle |^2 = \sum_{k=1}^{n[D/\varepsilon]}  |\langle (\Phi_k^\varepsilon)'(1),b\rangle |^2 \left( |a_k^\varepsilon|^2 + \frac{|b_k^\varepsilon|^2}{\mu_{k}^\varepsilon} \right). \]
Let us re-index the above summation as follows:
\[\sum_{k=1}^{n[D/\varepsilon]}  |\langle (\Phi_k^\varepsilon)'(1),b\rangle |^2 \left( |a_k^\varepsilon|^2 + \frac{|b_k^\varepsilon|^2}{\mu_{k}^\varepsilon} \right)=\sum_{k=1}^{[D/\varepsilon]}\sum_{l=1}^n|\langle (\Phi_{k,l}^\varepsilon)'(1),b\rangle |^2 \left( |a_{k,l}^\varepsilon|^2 + \frac{|b_{k,l}^\varepsilon|^2}{\mu_{k,l}^\varepsilon} \right).\]
Note that, \eqref{eq_evec_1} gives us that $\Phi^\varepsilon_{k,l} = \varphi_k^\varepsilon\cdot v_l$. Then 
\[\langle (\Phi_{k,l}^\varepsilon)'(1),b\rangle |^2 = |(\varphi_k)'(1)|^2| \langle b,v_l\rangle|=|(\varphi_k)'(1)|^2.\]
We recall the estimate \cite[Proposition 6.9]{CC99}
\[C_1 \lambda_k^\varepsilon \le |(\varphi_k^\varepsilon)'(1)|^2 \le C_2 \lambda_k^\varepsilon, \]
where $\lambda_k^\varepsilon = \int_0^1 a^\varepsilon(x) |(\varphi_k^\varepsilon)'(x)|^2 \rd x$.
Since $\mu_{k,l}^\varepsilon = \lambda_k^\varepsilon + \theta_l$, with $\theta_l$ bounded, we deduce that
\begin{align}\label{1303}
\begin{split}
\sum_{k=-n[D/\varepsilon]}^{n[D/\varepsilon]} |c_k^\varepsilon|^2 |\langle (\Phi_k^\varepsilon)'(1),b\rangle|^2 & \asymp \sum_{k=1}^{[D/\varepsilon]} \sum_{k=1}^l
\lambda_k^\varepsilon \left( |a_{k,l}^\varepsilon|^2 + \frac{|b_{k,l}^\varepsilon|^2}{\mu_{k,l}^\varepsilon} \right) \\
& \asymp \sum_{k=1}^{[D/\varepsilon]} \sum_{l=1}^n \left( \lambda_k^\varepsilon |a_{k,l}^\varepsilon|^2 + |b_{k,l}^\varepsilon|^2 \right) = \|(w^0, w^1) \|^2_{H_0^1(0,1) \times L^2(0,1)},
\end{split}
\end{align}
where we have used the following facts
\begin{align*}
(i)\ \|w^0\|_{(H_0^1(0,1))^n}^2 & = \sum_{n=1}^{n[D/\varepsilon]} |a^\varepsilon_k|^2\|\p_x \Phi^\varepsilon_k\|_{(L^2(0,1))^n}^2 = \sum_{n=1}^{[D/\varepsilon]} \sum_{k=1}^l |a^\varepsilon_{k,l}|^2 \|\p_x\Phi^\varepsilon_{k,l} \|_{(L^2(0,1))^n}^2\\
& = \sum_{n=1}^{[D/\varepsilon]} \sum_{k=1}^l|a^\varepsilon_{k,l}|^2 \|\p_x\varphi^\varepsilon_{k} \|_{L^2(0,1)}^2|v_l|_{\mathbb{R}^n}^2 \asymp \sum_{k=1}^{[D/\varepsilon]} \sum_{l=1}^n  \lambda_k^\varepsilon| a^\varepsilon_{k,l}|^2,\\
(ii)\ \|w^1\|_{(L^2(0,1))^n}^2 & = \sum_{k=1}^{n[D/\varepsilon]} |b^\varepsilon_{k}|^2= \sum_{k=1}^{[D/\varepsilon]} \sum_{l=1}^n|b^\varepsilon_{k,l}|^2 \asymp \sum_{k=1}^{[D/\varepsilon]} \sum_{l=1}^n \frac{\lambda_k}{\mu_{k,l}} |b^\varepsilon_{k,l}|^2.
\end{align*}
Here, the re-indexing follows according to the spectrum re-indexing as discussed before. Substituting \eqref{1303} into \eqref{1628} yields the desired estimate.
\end{proof}

Now we provide the analogous results for the high frequency region. For $\varepsilon>0,$  and  for some $N\geq 1$, define the following
\[H^{[D\varepsilon^{-1-1/N}]} = \text{span}\{\Phi_{k,i}^\varepsilon: k > [D\varepsilon^{-1-1/N}], i=1,2,\dots,n\}. \]

\begin{theorem} \label{thm_highfreq}
Let $a\in W^{N+1,\infty}(\mathbb{R})$ and $T>2 \ds \int_0^1\sqrt{a(x)} \rd x$. Then for any $(u^0,u^1) \in (L^2(0,1))^n\times (H^{-1}(0,1))^n$  there exists a control $f^\varepsilon\in L^2(0,T)$  such that the corresponding solution $u^\varepsilon$ verifies 
$$\left( \Pi_{ H^{[D\varepsilon^{-1-1/N}]}} u^\varepsilon(T,\cdot) , \Pi_{H^{ [D\varepsilon^{-1-1/N}] }} \p_t u^\varepsilon(T,\cdot)\right)=(0,0).$$
Moreover, there exists a constant $C := C(T) >0$, independent of $\varepsilon$, such that 
\begin{align*}
\| f^\varepsilon \|_{(L^2((0,T))^m} \leqslant C(T) \left\|\left( \Pi_{H^{[D\varepsilon^{-1-1/N}]}}{u}^0,\Pi_{H^{[D\varepsilon^{-1-1/N}]}}{u}^1\right) \right\|_{(L^2(0,1))^n\times (H^{-1}(0,1))^n},\\
\left\|\left(\Pi_{H^{[D\varepsilon^{-1-1/N}]}}{u}^\varepsilon,\Pi_{H^{[D\varepsilon^{-1-1/N}]}} \p_t u^\varepsilon \right) \right\|_{L^\infty(0,T;(L^2(0,1))^n\times (H^{-1}(0,1))^n)}\leq C(T)\|f^\varepsilon\|_{L^2(0,T)}.
\end{align*}
As $\varepsilon\to 0,$ the following convergences hold
\begin{align*}
\left(\Pi_{H^{[D\varepsilon^{-1-1/N}]}} u^\varepsilon, \Pi_{H^{ [D\varepsilon^{-1-1/N}]}} \p_t u^\varepsilon\right)&\rightharpoonup (0,0), \text{ weakly in } L^\infty(0,T;(L^2(0,T))^n \times (H^{-1}(0,1))^n),\\
f^\varepsilon & \rightharpoonup 0 , \quad \ \ \text{ weakly in }(L^2(0,T)).
\end{align*}
\end{theorem}
The proof of the above theorem follows the same strategy as in the low-frequency case (Proposition \ref{lowfreobs}). The key ingredient in the argument is the uniform estimate on eigenfunctions provided by Proposition \ref{eigenfunestihigh}. This allows one to reproduce the spectral decomposition argument and derive the corresponding observability inequality for initial data in $H^{D\varepsilon^{-1-1/N}}$, with a constant independent of $\varepsilon$. Once the uniform observability estimate is established, the uniform boundedness of the associated null controls with respect to $\varepsilon$ follows by duality. The remaining technical estimates can be obtained by adapting standard arguments, see, for instance, \cite{CC99}.

\section{Uniform bounds for the control and the projected state} \label{sec_uniform}
Now we will prove that the controls obtained in the previous section have uniform norm with respect to $\varepsilon$.
\begin{proof}[Proof of Theorem \ref{thm_control_m}]
Our aim is to prove null controllability of the projected solution using a duality argument. For this purpose, we consider the adjoint system
\begin{align}\label{1041}
\begin{cases}
\p_{tt} w - \mc{L}^\varepsilon w + A w = 0,  & (t,x)\in (0,T)\times (0,1), \\
w(t,0) = w(t,1) = 0, & t\in (0,T),\\
w(0,x) = w^0(x), \quad \p_t w(0,x) = w^1(x), & x\in (0,1),
\end{cases}
\end{align}
where $w^0, w^1 \in L_{[D/\varepsilon]}$. Multiplying \eqref{general_B_M} with $w$, then integrating on $[0,T] \times [0,1]$, and using integration by parts, we get
\begin{align*}
0 & = \int_0^T \int_0^1 \langle \p_{tt} u^\varepsilon - \mc{L}^\varepsilon u^\varepsilon + A u^\varepsilon , w\rangle \\
& = \int_0^T \int_0^1 \langle u^\varepsilon,  \p_{tt} w - \mc{L}^\varepsilon w + A w\rangle - \langle \p_t w(T,x),u^\varepsilon(T,x)\rangle_{H^{-1},H_0^1} +\langle  \p_t w(0,x), u^\varepsilon(0,x)\rangle_{H^{-1},H_0^1}  \\
& \qquad + \langle \p_t u^\varepsilon(T,x), w(T,x)\rangle - \langle \p_t u^\varepsilon(0,x), w(0,x)\rangle  \\
& \qquad + \int_0^T [a^\varepsilon(1) u^\varepsilon(t,1) w_x(t,1) - a^\varepsilon(0) u^\varepsilon(t,0) \p_x w(t,0)] \rd t \\
& \qquad - \int_0^T [ a^\varepsilon(1) \p_x u^\varepsilon(t,1) w(t,1) - a^\varepsilon(0) \p_x u^\varepsilon (t,0) w(t,0) ] \rd t,
\end{align*}
where $\langle \cdot, \cdot \rangle$ denotes the inner product in $\mathbb{R}^n$. Now, using the boundary and initial conditions given in \eqref{general_B_M} and \eqref{1041}, we get that
\begin{align*}
0 & = \langle u^0, w^1 \rangle_{L^2, L^2} - \langle u^1,w^0 \rangle_{H^{-1}, H^1_0}  + \int_0^T a^\varepsilon(1) \langle \p_x w(t,1), b f(t) \rangle_{\mb{R}^n} \rd t \\
& \qquad + \langle \p_t u^\varepsilon (T,x), w(T,x)\rangle_{H^{-1},H_0^1} - \langle u^\varepsilon(T,x),\p_t w(T,x) \rangle _{H^{-1},H_0^1}. \numberthis \label{eq_cont_cond_f}
\end{align*}
Taking inspiration from the above equation, we define the functional given by
\begin{align*}
J^\varepsilon & : (H_0^1(0,1))^n \times (L^2(0,1))^n \to \mb{R}, \\
J^\varepsilon(w^0, w^1) & = \langle u^0, w^1 \rangle_{L^2,L^2} - \langle u^1,w^0 \rangle_{H^{-1},H_0^1} + \frac{1}{2} \int_0^T a^\varepsilon(1) |\p_x w(t,1)|^2  \rd t,
\end{align*}
where $(w^0, w^1) \in L_{[D/\varepsilon]}$ and $w$ is the solution of \eqref{1041} associated with $(w^0,w^1)$. Let $(\hat{w}^\varepsilon_0, \hat{w}^\varepsilon_1)$ be a minimiser of $J^\varepsilon$, and let $\hat{w}^\varepsilon$ be the corresponding solution of \eqref{1041}. Using the first-order optimality condition by evaluating the derivative of $J^\varepsilon$ at the minimiser equal to zero, we get that 
\[\langle u^0, w^1 \rangle_{L^2(0,1) \times L^2(0,1)} - \langle u^1, w^0 \rangle_{H^{-1}(0,1)\times H_0^1(0,1)} + \int_0^T a^\varepsilon(1) \langle \p_x w(t,1), \p_x \hat{w}^\varepsilon(t,1) \rangle \rd t = 0,\]
for all $ (w^0, w^1) \in L_{[D/\varepsilon]}$.
Then let us define the control as follows
\begin{equation} \label{eq_control_def}
b f^\varepsilon(t) = \p_x \hat{w}^\varepsilon (t,1).
\end{equation}
Substituting this choice into \eqref{eq_cont_cond_f}, we obtain
\begin{align*}
\langle \p_t u^\varepsilon (T,x), w(T,x) \rangle_{H^{-1}(0,1)\times H_0^1(0,1)}
- \int_0^1 u^\varepsilon(T,x) \p_t w(T,x) \rd x = 0,
\end{align*}
holds for all $w^0,w^1 \in L_{[D/\varepsilon]}$. Now, time reversibility of wave equations implies that the above holds for all $w(T,\cdot),\p_t w(T,\cdot) \in L_{[D/\varepsilon]}$.
Then, we conclude that
\[ \Pi_{[D/\varepsilon]} u^\varepsilon(T,\cdot) = 0, \quad \Pi_{[D/\varepsilon]} \p_t u^\varepsilon (T,\cdot) = 0, \]
which proves \eqref{eq_proj_cont_low}, the null controllability of the projected solution.

Now, for each fixed $\varepsilon > 0$, the functional $J^\varepsilon$ attains its minimum at $(\hat{w}_0^\varepsilon, \hat{w}_1^\varepsilon)$. Hence,
\[J^\varepsilon(\hat{w}_0^\varepsilon, \hat{w}_1^\varepsilon) \leqslant J^\varepsilon(0,0) = 0,\]
which implies that
\[\frac{1}{2} \int_0^T a^\varepsilon(1) |\p_x \hat{w}^\varepsilon(t,1)|^2 \rd t \leqslant -\langle  u^0, \hat{w}_1^\varepsilon \rangle + \langle u^1, \hat{w}_0^\varepsilon \rangle. \]
Using duality, we obtain
\[ \frac{1}{2} \int_0^T a^\varepsilon(1) |\p_x \hat{w}^\varepsilon(t,1)|^2 \rd t
\le \|(u^0, u^1)\|_{L^2 \times H^{-1}} \|(\hat{w}_0^\varepsilon, \hat{w}_1^\varepsilon)\|_{H_0^1 \times L^2}. \]
By the observability inequality, there exists a constant $C(T) > 0$ such that
\[ \|(\hat{w}_0^\varepsilon, \hat{w}_1^\varepsilon)\|_{H_0^1 \times L^2}^2 \le C(T) \int_0^T a^\varepsilon(1)\, |\langle \p_x w^\varepsilon(t,1), b \rangle|^2 \rd t.\]
Combining the above inequalities yields
\[\int_0^T a^\varepsilon(1)\, |\p_x \hat{w}^\varepsilon(t,1)|^2 \rd t \le C(T)\, \|(u^0, u^1)\|_{L^2 \times H^{-1}}^2. \]
This proves that $\{\p_x \hat{w}^\varepsilon (t,1)\}_{\varepsilon \in (0,1)}$ is uniformly bounded in $L^2(0,T)$. Now, using \eqref{eq_control_def} shows that
\[ f^\varepsilon (t) = \frac{\langle \p_x \hat{w}^\varepsilon (t,1), b \rangle}{\| b \|^2}. \]
Thus, $\{ f^\varepsilon \}_{\varepsilon \in (0,1)}$ is uniformly bounded in $L^2(0,T)$.

\noindent\textbf{Claim I.}
Let $\eta \in C^1(0,T; L_{[D/\varepsilon]})$ and let $\psi^\varepsilon$ solve
\begin{align}\label{1210}
\begin{cases}
\p_{tt} \psi^\varepsilon - \mc{L}^\varepsilon \psi^\varepsilon + A \psi^\varepsilon = \eta, & (t,x)\in (0,T)\times (0,1),\\
\psi^\varepsilon(t,0) = \psi^\varepsilon(t,1) = 0, & t \in (0,T),\\
\psi^\varepsilon(0,x) = 0, \quad \p_t \psi^\varepsilon(0,x) = 0, & x \in (0,1).
\end{cases}
\end{align}
Then there exists a constant $C>0$, independent of $\varepsilon$, such that
\[ \|\p_x \psi^\varepsilon(1,\cdot)\|_{L^2(0,T)} \le C \|\eta\|_{L^1(0,T; L^2(0,1))}.\]
\medskip
\noindent\textit{Proof of the claim.}
Using Duhamel's principle, we write
\[\psi^\varepsilon(t,x) = \int_0^t \tilde{\psi}^\varepsilon(t-s,x,s) \rd s,\]
where $\tilde{\psi}^\varepsilon$ solves 
\[ \begin{cases}
\p_{tt} \tilde{\psi}^\varepsilon - \mc{L}^\varepsilon \tilde{\psi}^\varepsilon + A\tilde{\psi}^\varepsilon = 0, & (t,x,s)\in (0,T)\times (0,1)\times (0,T),\\
\tilde{\psi}^\varepsilon(t,0,s)=\tilde{\psi}^\varepsilon(t,1,s)=0, & (t,s)\in (0,T)\times (0,T),\\
\tilde{\psi}^\varepsilon(0,x,s) = 0, \quad \p_t \tilde{\psi}^\varepsilon(0,x,s) = \eta(s,x) & (x,s)\in (0,1)\times (0,T).
\end{cases} \]
By the observability inequality, there exists a constant $C>0$, independent of $\varepsilon$, such that
\begin{align}\label{1713}
\|\p_x \tilde{\psi}^\varepsilon(1,\cdot,s)\|_{L^2(0,T)} \le C \|\eta(s,\cdot)\|_{L^2(0,1)}.
\end{align}
Using Minkowski's integral inequality, we obtain
\begin{align*}
\|\p_x \psi^\varepsilon(\cdot,1)\|_{L^2(0,T)} &\le \int_0^T |\p_x \tilde{\psi}^\varepsilon (t-s,1,s) \|_{L^2(s,T)} \rd s \le \int_0^T \|\p_x \tilde{\psi}^\varepsilon(\cdot,1,s)\|_{L^2(0,T)} \rd s.
\end{align*}
Using \eqref{1713}, it follows that
\[\|\p_x \psi^\varepsilon(\cdot,1)\|_{L^2(0,T)} \le C \int_0^T |\eta(s,\cdot) \|_{L^2(0,1)} \rd s,\]
which yields
\[\|\p_x \psi^\varepsilon(1,\cdot)\|_{L^2(0,T)}\le C \|\eta\|_{L^1(0,T; L^2(0,1))}.\]
This completes the proof of the claim.

\medskip

\noindent\textbf{Claim II.}
Let $u^\varepsilon$ be defined by transposition:
\[ \int_0^T \int_0^1 u^\varepsilon \eta \rd x \rd t= - \int_0^T a^\varepsilon(1)\, f^\varepsilon(t)\, \p_x \psi^\varepsilon(t,1) \rd t, \quad \forall \eta \in C^1(0,T; L_{[D/\varepsilon]}).
\]
Then
\begin{align*}
\|\Pi_{L_{[D/\varepsilon]}} u^\varepsilon\|_{L^\infty(0,T; L^2(0,1))} & \le C \|f^\varepsilon\|_{L^2(0,T)}, \\
\|\Pi_{L_{[D/\varepsilon]}} \p_t u^\varepsilon\|_{L^\infty(0,T;H^{-1}(0,1))} & \le C \|f^\varepsilon\|_{L^2(0,T)}, \numberthis \label{1226}
\end{align*}
where $C$ is independent of $\varepsilon$.
\medskip

\noindent\textit{Proof of the claim.}
Using the transposition identity together with H\"older's inequality, we obtain
\[ \left| \int_0^T \int_0^1 u^\varepsilon \eta \rd x \rd t \right| \le C \|f^\varepsilon\|_{L^2(0,T)} \|\p_x \psi^\varepsilon(\cdot,1)\|_{L^2(0,T)}. \]
By Claim I, we have
\[\left| \int_0^T \int_0^1 u^\varepsilon \eta \rd x \rd t \right| \le C \|f^\varepsilon\|_{L^2(0,T)} \|\eta\|_{L^1(0,T; L^2(0,1))}. \]
Since the above identity holds for all $\eta \in L^2(0,T; L_{[D/\varepsilon]})$, it follows that $\Pi_{L_{[D/\varepsilon]}} u^\varepsilon$ defines a bounded linear functional on $L^1(0,T; L^2(0,1))$. Hence,
\[ \Pi_{L_{[D/\varepsilon]}} u^\varepsilon \in L^\infty(0,T; L^2(0,1)) \]
and
\[ \|\Pi_{L_{[D/\varepsilon]}} u^\varepsilon\|_{L^\infty(0,T; L^2(0,1))}
\le C \|f^\varepsilon\|_{L^2(0,T)}. \]
To derive \eqref{1226}, we formally replace $\eta$ by $\p_t \eta$ in \eqref{1210}. Then the transposition solution satisfies
\[ \int_0^T \int_0^1 u^\varepsilon \p_t \eta \rd x \rd t = - \int_0^T a^\varepsilon(1) f^\varepsilon(t) \p_x \psi^\varepsilon(t,1) \rd t.\]
Proceeding as above, we obtain
\[
\|\Pi_{L_{[D/\varepsilon]}} \p_t u^\varepsilon \|_{L^\infty(0,T;H^{-1}(0,1))}
\le C \|f^\varepsilon\|_{L^2(0,T)}.
\]
This completes the proof of the claim and also proof of Theorem \ref{thm_control_m}.
\end{proof}

\begin{proof}[Proof of Theorem \ref{thm_homogen_m}]
The proof follows the asymptotic strategy developed in~\cite{CC99}. From the uniform bounds established in Theorem \ref{thm_control_m}, we deduce that
\begin{equation*}
\|f^\varepsilon\|_{L^2(0,T)} \le C,
\end{equation*}
\begin{equation*}
\left\|\left(\Pi_{L_{[D/\varepsilon]}} u_\varepsilon, \Pi_{L_{[D/\varepsilon]}} \partial_t u_\varepsilon\right)\right\|_{L^\infty (0,T;(L^2(0,1))^n \times (H^{-1}(0,1))^n )} \le C,
\end{equation*}
where the constant $C > 0$ is independent of $\varepsilon$. Consequently, by standard compactness arguments, there exists a subsequence (still denoted by $\varepsilon$) as well as limit functions $f_0 \in L^2(0,T)$ and $u$ such that the following weak and weak-$*$ convergences hold as $\varepsilon \to 0$:
\begin{equation*}
a^\varepsilon(1)f^\varepsilon \rightharpoonup a^0f_0 \quad \text{weakly in } L^2(0,T),
\end{equation*}
\begin{equation*}
\left(\Pi_{L_{[D/\varepsilon]}} u_\varepsilon,\Pi_{L_{[D/\varepsilon]}} \partial_t u_\varepsilon\right) {\rightharpoonup} (u,\partial_t u) \quad \text{ weakly-* in } L^\infty\bigl(0,T; L^2(0,1)\times H^{-1}(0,1)\bigr).
\end{equation*}

To identify the limit system, we invoke the formulation for the method of transposition as in \cite[Theorem 4]{CC99}. Passing to the limit via standard homogenization techniques, we obtain the effective homogenized system.

Finally, by virtue of the \emph{Hilbert Uniqueness Method} (HUM), the null control minimizing the $L^2(0,T)$ cost functional for the homogenized problem is uniquely determined. Due to this uniqueness of the limit, a standard compactness argument implies that the entire sequence $\{ f^\varepsilon \}$ converges weakly to $f_0$, rather than merely a subsequence. This completes the proof.
\end{proof}

\section{Controllability of all initial data in \texorpdfstring{$L^2 \times H^{-1}$}{L² × H⁻¹} with added feedback control} \label{sec_feedback}

In this section, we prove that the system can be driven to rest from any initial datum in the space $L^2(0,1) \times H^{-1}(0,1)$ by means of an additional feedback-type control. The strategy is to perform a change of coordinates that transforms the original problem into an equivalent one with constant coefficients, for which the controllability problem is well understood.

We first recall the relation between two formulations of the wave equation with variable coefficients. Consider the equations
\begin{equation*} 
\partial_{tt} u^\varepsilon - \partial_x(a^\varepsilon(x) \partial_x u^\varepsilon) + A u^\varepsilon = 0,
\end{equation*}
and
\begin{equation} \label{eq_related_2}
\rho^\varepsilon \partial_{tt} u^\varepsilon - \partial_{xx} u^\varepsilon + \rho^\varepsilon A u^\varepsilon + \frac{\partial_x \rho^\varepsilon}{\rho^\varepsilon} \partial_x u^\varepsilon = 0.
\end{equation}
These two formulations are equivalent under the relation
\[
\rho^\varepsilon(x) a^\varepsilon(x) = 1.
\]
Indeed, using this identity, we obtain
\[
\partial_x a^\varepsilon = -\frac{\partial_x \rho^\varepsilon}{(\rho^\varepsilon)^2}, \qquad \partial_x \rho^\varepsilon = -\frac{\partial_x a^\varepsilon}{(a^\varepsilon)^2},
\]
and a direct computation yields
\begin{align*}
\partial_{tt} u^\varepsilon - \partial_x (a^\varepsilon(x) \partial_x u^\varepsilon) + A u^\varepsilon
&= \partial_{tt} u^\varepsilon - \partial_x \left( \frac{\partial_x u^\varepsilon}{\rho^\varepsilon} \right) + A u^\varepsilon \\
&= \partial_{tt} u^\varepsilon - \frac{\partial_{xx} u^\varepsilon}{\rho^\varepsilon}
+ \frac{\partial_x \rho^\varepsilon}{(\rho^\varepsilon)^2} \partial_x u^\varepsilon + A u^\varepsilon \\
&= \frac{1}{\rho^\varepsilon} \left( \rho^\varepsilon \partial_{tt} u^\varepsilon
- \partial_{xx} u^\varepsilon
+ \frac{\partial_x \rho^\varepsilon}{\rho^\varepsilon} \partial_x u^\varepsilon
+ \rho^\varepsilon A u^\varepsilon \right).
\end{align*}
Since the assumption \eqref{upper_lower_bound_1} ensures that $a^\varepsilon$ is bounded above and below by positive constants, the same holds for $\rho^\varepsilon$. Hence the two equations are equivalent in the sense of well-posedness and controllability. From now on, we shall work with the formulation \eqref{eq_related_2}. Our main result in this section is the following.

\begin{theorem} \label{thm_full_space}
Let $(u_0,u_1) \in L^2(0,1) \times H^{-1}(0,1)$ and let $\rho^\varepsilon \in W^{1,\infty}(\mathbb{R})$ be a $1$-periodic function satisfying
\[
0 < \rho_m \leqslant \rho^\varepsilon(x) \leqslant \rho_M < \infty, \quad \text{a.e. } x \in \mathbb{R}.
\]
Assume $T > 2 \displaystyle\int_0^1 \sqrt{\rho^\varepsilon(s)} \, \mathrm{d}s$ and let $\varepsilon > 0$ be sufficiently small. Then there exists a control function $f^\varepsilon \in L^2(0,T)$ such that the solution of
\begin{equation} \label{full_space}
\begin{cases}
2 \rho^\varepsilon \partial_{tt} u^\varepsilon - \partial_{xx} u^\varepsilon + 2 \rho^\varepsilon A u^\varepsilon = \rho^\varepsilon \partial_x \left( \dfrac{1}{\rho^\varepsilon} \partial_x u^\varepsilon \right), & (t,x)\in (0,T)\times(0,1), \\[1.2ex]
u^\varepsilon(t,0)=0,\qquad u^\varepsilon(t,1)= b f^\varepsilon(t), & t \in (0,T),\\[0.5ex]
u^\varepsilon(0,x)=u^0(x),\qquad \partial_t u^\varepsilon(0,x)=u^1(x), & x \in (0,1),
\end{cases}
\end{equation}
satisfies the terminal condition
\[u^\varepsilon(T,\cdot) = \partial_t u^\varepsilon(T,\cdot) = 0.\]
\end{theorem}

\begin{proof}
Observe that the first equation in \eqref{full_space} can be rewritten as
\begin{align}
0 &= 2 \rho^\varepsilon \partial_{tt} u^\varepsilon - \partial_{xx} u^\varepsilon + 2 \rho^\varepsilon A u^\varepsilon
- \rho^\varepsilon \partial_x \left( \frac{1}{\rho^\varepsilon} \partial_x u^\varepsilon \right) \notag \\
&= 2 \rho^\varepsilon \partial_{tt} u^\varepsilon - \partial_{xx} u^\varepsilon + 2 \rho^\varepsilon A u^\varepsilon
+ \frac{\partial_x \rho^\varepsilon}{\rho^\varepsilon} \partial_x u^\varepsilon - \partial_{xx} u^\varepsilon \notag \\
&= 2 \rho^\varepsilon \partial_{tt} u^\varepsilon - 2 \partial_{xx} u^\varepsilon + 2 \rho^\varepsilon A u^\varepsilon
+ \frac{\partial_x \rho^\varepsilon}{\rho^\varepsilon} \partial_x u^\varepsilon. \label{eq_transf_2}
\end{align}
Dividing by $2$, we obtain
\[
\rho^\varepsilon \partial_{tt} u^\varepsilon - \partial_{xx} u^\varepsilon + \rho^\varepsilon A u^\varepsilon
+ \frac{1}{2\rho^\varepsilon} \partial_x \rho^\varepsilon \partial_x u^\varepsilon = 0,
\]
or equivalently,
\[
\rho^\varepsilon \partial_{tt} u^\varepsilon - \partial_{xx} u^\varepsilon + \rho^\varepsilon A u^\varepsilon
+ \frac{1}{\rho^\varepsilon} \partial_x \rho^\varepsilon \partial_x u^\varepsilon
= \frac{1}{2\rho^\varepsilon} \partial_x \rho^\varepsilon \partial_x u^\varepsilon.
\]
Thus, the term on the right-hand side plays the role of a feedback control in the original system.

We now claim that system \eqref{full_space} is equivalent to the system
\begin{equation} \label{eq_sqrt_rho}
\sqrt{\rho^\varepsilon} \partial_{tt} u^\varepsilon
- \partial_x \left( \frac{\partial_x u^\varepsilon}{\sqrt{\rho^\varepsilon}} \right)
+ \sqrt{\rho^\varepsilon} A u^\varepsilon = 0.
\end{equation}
Indeed, expanding the second term gives
\[
\sqrt{\rho^\varepsilon} \partial_{tt} u^\varepsilon
+ \frac{1}{2(\rho^\varepsilon)^{3/2}} \partial_x \rho^\varepsilon \partial_x u^\varepsilon
- \frac{\partial_{xx} u^\varepsilon}{\sqrt{\rho^\varepsilon}}
+ \sqrt{\rho^\varepsilon} A u^\varepsilon = 0.
\]
Multiplying by $2\sqrt{\rho^\varepsilon}$ yields precisely \eqref{eq_transf_2}, which proves the claim.

The advantage of \eqref{eq_sqrt_rho} is that it becomes constant-coefficient under a suitable change of variables. Define
\[
y(x) = \int_0^x \sqrt{\rho^\varepsilon(s)} \, \mathrm{d}s.
\]
Then $\partial_x y = \sqrt{\rho^\varepsilon(x)}$, and consequently
\[
\partial_x u^\varepsilon = \sqrt{\rho^\varepsilon} \, \partial_y u^\varepsilon,
\qquad
\partial_x \left( \frac{\partial_x u^\varepsilon}{\sqrt{\rho^\varepsilon}} \right)
= \partial_x(\partial_y u^\varepsilon)
= \sqrt{\rho^\varepsilon} \, \partial_{yy} u^\varepsilon.
\]
Substituting into \eqref{eq_sqrt_rho} and dividing by $\sqrt{\rho^\varepsilon}>0$, we obtain the constant-coefficient wave equation
\[
\partial_{tt} u^\varepsilon - \partial_{yy} u^\varepsilon + A u^\varepsilon = 0,
\]
posed on the spatial interval
\[
y \in \left(0, L_\varepsilon\right), \qquad L_\varepsilon := \int_0^1 \sqrt{\rho^\varepsilon(s)} \, \mathrm{d}s.
\]
Thus, the original control system \eqref{full_space} is transformed into the following equivalent problem:
\begin{equation} \label{full_space_2}
\begin{cases}
\partial_{tt} u^\varepsilon - \partial_{yy} u^\varepsilon + A u^\varepsilon = 0, & (t,y)\in (0,T)\times(0,L_\varepsilon), \\[0.5ex]
u^\varepsilon(t,0)=0,\qquad u^\varepsilon(t,L_\varepsilon)= b f^\varepsilon(t), & t \in (0,T),\\[0.5ex]
u^\varepsilon(0,y)=u^0(y),\qquad \partial_t u^\varepsilon(0,y)=u^1(y), & y \in (0,L_\varepsilon).
\end{cases}
\end{equation}
The controllability of this constant-coefficient system, with initial data lying in the space $L^2(0,L_\varepsilon) \times H^{-1}(0,L_\varepsilon)$, is classical; see, for instance, \cite{avd_ter}. Since the transformation $y \mapsto x$ is a diffeomorphism, the control $f^\varepsilon$ obtained for \eqref{full_space_2} yields the desired controllability for the orig inal system.
\end{proof}

Finally, we establish the convergence of the controls and the corresponding solutions as the parameter $\varepsilon \to 0$. Let
\[
\overline{\sqrt{\rho}} := \int_0^1 \sqrt{\rho(s)} \, \mathrm{d}s.
\]
We have the following compactness result.

\begin{theorem} \label{thm_full_space_hom}
Under the hypotheses of Theorem \ref{thm_full_space}, for every $\varepsilon>0$ there exists a control $f^\varepsilon \in L^2(0,T)$ for system \eqref{full_space} such that, up to a subsequence,
\begin{align*} 
u^\varepsilon &\longrightarrow u, && \text{strongly in } L^\infty(0,T;L^2(0,1)), \\
f^\varepsilon &\longrightarrow f, && \text{strongly in } L^2(0,T),
\end{align*}
where $(u,f)$ satisfies the homogenized system
\begin{equation*} 
\begin{cases}
(\overline{\sqrt{\rho}})^2 \partial_{tt} u - \partial_{xx} u + (\overline{\sqrt{\rho}})^2 A u = 0, & (t,x)\in (0,T)\times(0,1), \\[0.5ex]
u(t,0)=0,\qquad u(t,1)= b f(t), & t \in (0,T),\\[0.5ex]
u(0,x)=u_0(x),\qquad \partial_t u(0,x)=u_1(x), & x \in (0,1),
\end{cases}
\end{equation*}
together with the terminal conditions
\[u(T,\cdot) = \partial_t u(T,\cdot) = 0.\]
\end{theorem}

\end{document}